\documentclass[11pt]{article}
\usepackage[T1]{fontenc} 
\usepackage{amsmath,amsthm,amssymb}
\usepackage{chemmacros} 
\usepackage{textgreek}  
\usepackage{geometry}
\usepackage{enumitem}   
\usepackage{hyperref}
\usepackage{mathtools}  

\hypersetup{colorlinks=true, linkcolor=blue, citecolor=blue, urlcolor=blue}
\theoremstyle{plain}
\newtheorem{theorem}{Theorem}[section]
\newtheorem{proposition}[theorem]{Proposition}
\newtheorem{lemma}[theorem]{Lemma}
\newtheorem{corollary}[theorem]{Corollary}
\theoremstyle{definition}
\newtheorem{example}[theorem]{Example}
\theoremstyle{remark}
\newtheorem{remark}[theorem]{Remark}

\newcommand{\D}{\mathbb{D}}

\begin{document}

\title{Transfinite Iteration of Operator Transforms and Spectral Projections in Hilbert and Banach Spaces}
\author{Faruk Alpay\thanks{Department of Mathematics, Lightcap Institute. \texttt{alpay@lightcap.ai}} \and Taylan Alpay\thanks{Department of Aerospace Engineering, Turkish Aeronautical Association.} \and Hamdi Alakkad\thanks{Department of Engineering, Bah\c{c}e\c{s}ehir University.}}
\date{}
\maketitle

\begin{abstract}
We develop an analytic framework for ordinal-indexed, multi-layer iterations of bounded linear operator transforms on Hilbert and reflexive Banach spaces. Given a normal operator $A$ on a Hilbert space with a family of $K$ polynomial or holomorphic transforms $\Phi_k(T) = p_k(T)$ satisfying $\sup_{\lambda\in\sigma(A)}|p_k(\lambda)| \le 1$, $p_k(1) = 1$, and with the \emph{peripheral spectrum} fixed under $p_k$, we prove that the transfinite iteration $A^{(\alpha)}$ (applying $\Phi_1,\dots,\Phi_K$ in cycles) converges in the strong operator topology by some countable stage $\alpha < \omega_1$ to an idempotent $P$ that commutes with $A$. The limit $P$ is the spectral projection of $A$ onto the joint fixed-point set and $\sigma(P) = \{0,1\}$. We also establish a precise spectral mapping property for every stage of the iteration under the functional calculus. In a reflexive Banach space setting, for a power-bounded Ritt operator (or sectorial generator of a bounded analytic semigroup) $A$ admitting a bounded $H^\infty$ functional calculus, an analogous multi-layer iteration $\Phi_K \circ \cdots \circ \Phi_1$ (with $\sup | \varphi_k| \le 1$, $\varphi_k(1) = 1$) converges in strong operator topology to a bounded projection $P$ commuting with $A$. The range of $P$ consists of vectors invariant under all transforms (in particular, $AP = PA$). We provide proofs via the spectral theorem (normal case) and via resolvent and ergodic techniques (Ritt/sectorial case), along with explicit examples and counterexamples. A Fej\'er-monotonicity argument yields an ordinal bound $\le \omega$ (countable) for stabilization. We construct cases achieving the extremal countable stage and show that allowing non-commuting layers or dropping the Ritt/sectorial assumptions can destroy convergence. Connections with classical mean ergodic projections~\cite{DunfordSchwartz1,DunfordSchwartz2} and stability theorems of Arendt--Batty and Lyubich--Vũ~\cite{DunfordSchwartz2,ArendtBatty1988,LyubichVu1988} are discussed.
\end{abstract}

\section{Introduction}\label{sec:intro}
Iterative methods for projecting onto invariant subspaces and spectral subspaces are a classical theme in operator theory and ergodic theory~\cite{DunfordSchwartz1,DunfordSchwartz2}. A basic instance is the power method: if $A$ is a contraction on a Hilbert space, then $A^n$ often converges (strongly) to a projection onto the fixed-point space $\ker(I - A)$ by the mean ergodic theorem. More generally, for a power-bounded operator with spectral radius $1$, classical results (e.g.\ the Katznelson--Tzafriri theorem and the Arendt--Batty--Lyubich--Vũ stability theorem) give conditions under which $A^n \to P$ as $n \to \infty$, where $P$ is the projection onto the eigenspace at eigenvalue $1$. In continuous time, similar phenomena occur: a $C_0$-semigroup $T(t) = e^{tA}$ may converge strongly to a projection as $t \to \infty$ under spectral assumptions on the generator $A$ (e.g.\ no spectrum on the imaginary axis except $0$). Such limits are \emph{ergodic projections} in the sense of mean ergodic theory.

In this work, we introduce a broad iteration scheme that generalizes these one-parameter ergodic results to \emph{multi-layer, transfinite iterations of operator transforms}. Concretely, we consider a family of $K$ operator functions $\Phi_1,\dots,\Phi_K$ (``layers'') and define an ordinal-indexed sequence of operators: $A^{(\alpha+1)} = \Phi_K(\cdots \Phi_1(A^{(\alpha)})\cdots)$ for successor ordinals, and at a limit ordinal $\lambda$ we define $A^{(\lambda)} = \text{SOT-}\lim_{\alpha\to\lambda} A^{(\alpha)}$ assuming the limit exists in the strong operator topology. This transfinite iterative process continues (if necessary) through countable ordinals. We show that under suitable analytic conditions on the $\Phi_k$, the sequence $(A^{(\alpha)})_{\alpha<\omega_1}$ stabilizes by some countable stage: there is an ordinal $\alpha_* < \omega_1$ such that $A^{(\alpha)} = A^{(\alpha_*)}$ for all $\alpha \ge \alpha_*$. In fact, we obtain explicit bounds on $\alpha_*$ in terms of the properties of the transforms (often $\alpha_* \le \omega$, the first countable ordinal). The limiting operator $A^{(\alpha_*)}$ is a bounded idempotent (projection) $P$ that commutes with the original operator $A$~\cite{ArendtBatty1988,LyubichVu1988}.

We treat two primary frameworks:

\textbf{(I) Hilbert space (normal $A$):} $A$ is a normal operator on a complex Hilbert space $H$. We assume $\Phi_k(T) = p_k(T)$ for $k=1,\dots,K$, where each $p_k$ is a polynomial or a bounded holomorphic function on an open set containing $\sigma(A)$. These $p_k$ are \textbf{Schur holomorphic functions} on the spectrum, meaning $\sup_{\lambda\in\sigma(A)}|p_k(\lambda)| \le 1$, and they fix the peripheral spectrum: in particular $p_k(1) = 1$, and for every $\lambda\in\sigma(A)$ with $|\lambda| = 1$, if $|p_k(\lambda)| = 1$ then $p_k(\lambda) = \lambda$. (This is the \emph{peripheral fixed-point property}.) Such $\Phi_k$ clearly commute with each other (since they are functions of $A$ or its iterates), so the multi-layer composition $\Phi_K \circ \cdots \circ \Phi_1$ applied to $A^{(\alpha)}$ is unambiguous.

\textbf{(II) Banach space (power-bounded $A$):} $A$ is a power-bounded linear operator on a reflexive Banach space $X$ (discrete-time case), or $-A$ generates a bounded analytic $C_0$-semigroup (continuous-time case). We assume $A$ admits a bounded $H^\infty$ functional calculus on the spectral domain (a Stolz disk containing $\sigma(A)$ in the discrete case, or a sector containing $\sigma(A)$ in the continuous case~\cite{HaaseFC,CDMY96}). We take $\Phi_k(T) = \varphi_k(T)$ where each $\varphi_k$ is a rational or holomorphic function on the domain of $A$ (unit disk or sector) such that $\sup|\varphi_k| \le 1$. These conditions ensure $\varphi_k(A)$ is well-defined (via the functional calculus) and $\|\varphi_k(A)\| \le C$ for some constant $C$, and again $\varphi_k$ fixes the point $1$ on the boundary of $\sigma(A)$ (the unit circle or spectral boundary). We consider the same transfinite iteration which is well-defined since each $\varphi_k$ preserves the class of operators (Ritt or sectorial) and the $\varphi_k(A^{(\alpha)})$ all commute (being functional calculus images of a common original $A$ or its iterates).

Under these assumptions, our main results can be summarized as follows.

\section{Convergence results in the Hilbert space (normal) case}\label{sec:normal-case}

\noindent \textbf{Standing assumptions (normal case).} Let $A \in B(H)$ be a normal operator on a Hilbert space $H$, with spectral measure $E_A$ and spectrum $\sigma(A)\subseteq U$ (where $U\subset\mathbb{C}$ is an open set). Let $p_1,\dots,p_K$ be functions holomorphic on $U$ (polynomials are a special case) with a uniform bound $\sup_{\lambda\in\sigma(A)}|p_k(\lambda)| \le 1$. Define the composite function $f := p_K \circ \cdots \circ p_1$, and denote by $f^m$ the $m$-fold composition ($f^m := f \circ f \circ \cdots \circ f$, $m$ times). We apply these to $A$ via the functional calculus: write $A^{(m)} = f^m(A)$ for $m\in\mathbb{N}$ (so $A^{(0)}=A$ and $A^{(m+1)} = f(A^{(m)}) = f^{m+1}(A)$). At the first limit stage $m \to \infty$, if the strong operator limit exists we denote $A^{(\omega)} = \lim_{m\to\infty} A^{(m)}$. Note that by the properties of the normal functional calculus, $A^{(m)} = (f^m)(A)$, and two standard properties will be used tacitly throughout:
\begin{itemize}
\item \textbf{(Composition in the calculus).} If $h$ is a bounded Borel function on $\sigma(A)$ and $g$ is a bounded Borel function on $h(\sigma(A))$, then $g(h(A)) = (g\circ h)(A)$.
\item \textbf{(Dominated convergence in the calculus).} If $g_n$ are bounded Borel functions on $\sigma(A)$ with $\sup_n\|g_n\|_\infty < \infty$ and $g_n(\lambda)\to g(\lambda)$ pointwise on $\sigma(A)$, then $g_n(A) \to g(A)$ in strong operator topology.
\end{itemize}
(Proofs of these facts can be found in Appendix~A as Lemmas~\ref{lem:dominated} and \ref{lem:composition} respectively.)

We now present two main theorems (labeled A' and B') for the Hilbert space case, concerning convergence to a projection (Theorem~A') and the spectral behavior along the iteration (Theorem~B'). These results are formulated under general conditions that will later be verified by more specific assumptions (see Lemma~\ref{lem:conv-to-01} and Corollary~\ref{cor:spec-projection} below).

\begin{theorem}[A' -- Convergence via function iteration $\implies$ spectral projection]\label{thm:aprime}
\emph{Suppose (under the standing setup above) that the sequence of functions $f^m$ is uniformly bounded and converges pointwise on $\sigma(A)$ to a (Borel) limit function $\chi:\sigma(A)\to\mathbb{C}$. That is, assume:
$\sup_m \|f^m\|_\infty < \infty$ and $\lim_{m\to\infty} f^m(\lambda) = \chi(\lambda)$ for all $\lambda\in\sigma(A)$.
Then the strong limit $A^{(\omega)}$ exists and $A^{(\omega)} = \chi(A)$.
If, moreover, $\chi$ takes only the values $0$ and $1$, then $P := \chi(A)$ is an orthogonal projection that commutes with $A$, and
$\mathrm{Ran}\,P = E_A(\chi^{-1}(\{1\}))H$.
(Indeed, $\sigma(P)=\{0,1\}$ unless $P=0$ or $P=I$.)}
\end{theorem}

\begin{proof}
By the dominated convergence property of the normal calculus, $f^m(A) = (f^m)(A) \to \chi(A)$ strongly as $m\to\infty$. If $\chi$ takes only the values $0$ and $1$, then $\chi(A) = \int \chi\,dE_A = E_A(\chi^{-1}(\{1\}))$, which is an orthogonal projection (since $\chi^{-1}(\{1\})$ is a measurable subset of $\sigma(A)$) and clearly commutes with $A$ (being a function of $A$). The range of $\chi(A)$ is $\mathrm{Ran}\,E_A(\chi^{-1}(\{1\}))$, which is exactly the subspace of $H$ where $A$ acts with spectrum in $\chi^{-1}(\{1\})$. This shows the stated description of $P$ and its range.
\end{proof}

\begin{theorem}[B' -- Spectral mapping along the iteration]\label{thm:bprime}
\emph{Under the same standing assumptions, one has:
\begin{itemize}
\item[\textup{(i)}] \textup{(Finite stages).} For each finite $m\ge 1$,
$\sigma(A^{(m)}) = \sigma(f^m(A)) = f^m(\sigma(A))$.
In other words, spectral mapping holds as an equality at every finite iterate (since $f^m$ is holomorphic on a neighborhood of $\sigma(A)$).
\item[\textup{(ii)}] \textup{(Limit stage).} If $f^m \to \chi$ pointwise on $\sigma(A)$ with $\sup_m \|f^m\|_\infty < \infty$ (so that $A^{(\omega)} = \chi(A)$ exists by Theorem~\ref{thm:aprime}), then
$\sigma(A^{(\omega)}) \subseteq \overline{\chi(\sigma(A))}$.
In particular, if $\chi\in\{0,1\}$ (so that $A^{(\omega)}=P$ is a projection), then $\sigma(A^{(\omega)}) \subseteq \{0,1\}$ (and in fact $\sigma(P) = \{0,1\}$ unless $P$ is trivial $0$ or $I$).
\end{itemize}}
\end{theorem}

\begin{proof}
Part (i) is immediate: since $f^m$ is holomorphic on an open set containing $\sigma(A)$, the spectral mapping theorem for normal operators gives $\sigma(f^m(A)) = f^m(\sigma(A))$ (no inclusion restriction is needed in the normal case).

For part (ii), we know $A^{(\omega)} = \chi(A)$ by Theorem~\ref{thm:aprime}. By functional calculus, $\sigma(\chi(A)) = \overline{\chi(\sigma(A))}$ if $\chi$ were continuous on $\sigma(A)$. However, $\chi$ might be just a bounded Borel function (for example, $\chi$ could be discontinuous if it jumps between $0$ and $1$ on different parts of the spectrum). In general, one only has $\sigma(\chi(A)) = \mathrm{ess\,range}_{E_A}(\chi)$, the essential range of $\chi$ with respect to the spectral measure $E_A$. Proposition~\ref{prop:ess-range} below provides a proof of this fact. Since the essential range $\mathrm{ess\,range}_{E_A}(\chi)$ is always a subset of the closure of the pointwise range $\overline{\chi(\sigma(A))}$, we have $\sigma(A^{(\omega)}) \subseteq \overline{\chi(\sigma(A))}$. Finally, if $\chi$ only takes values $0,1$, then $\chi(\sigma(A))\subseteq\{0,1\}$ and hence $\sigma(P)\subseteq\{0,1\}$. (In fact, as noted, $P=\chi(A)$ is a nonzero projection unless $A^{(\omega)}$ turned out trivial $0$ or $I$, so $\sigma(P)$ must equal $\{0,1\}$.)
\end{proof}

We isolate here the general fact used above, for completeness:

\begin{proposition}[Essential-range description of $\sigma(\chi(A))$]\label{prop:ess-range}
Let $A$ be a normal operator with spectral measure $E_A$, and let $\chi:\sigma(A)\to\mathbb{C}$ be a bounded Borel function. Then
$\sigma(\chi(A)) = \{\lambda \in \mathbb{C} : E_A(\{\omega \in \sigma(A) : |\chi(\omega) - \lambda| < \varepsilon\}) \ne 0 \text{ for all } \varepsilon > 0\}$.
\end{proposition}

\begin{proof}
In the spectral representation (via the Spectral Theorem), $A$ is unitarily equivalent to a multiplication operator $M_a$ on some $L^2$ space such that $A \cong M_a$ and $\chi(A) \cong M_{\chi\circ a}$. In this model, $\sigma(\chi(A))$ is the essential range of $\chi\circ a$ with respect to the underlying measure. But since the measure associated with $A$ is precisely $\mu_x(B) = \langle E_A(B)x,x\rangle$ (for $x$ in an orthonormal basis), the essential range of $\chi\circ a$ equals the set given in the statement (which is defined in terms of $E_A$)~\cite{BadeaEtAl,Davidson}. A direct operator-theoretic proof can be given as follows. If $\lambda$ lies outside the essential range of $\chi$, then there exists $\delta>0$ such that $|\chi(\omega)-\lambda|\ge \delta$ for $E_A$-almost all $\omega$. This implies $(\chi-\lambda)^{-1}$ is essentially bounded on $\sigma(A)$ and thus defines a bounded Borel function. By functional calculus, $(\chi(A) - \lambda I)$ then has a bounded two-sided inverse $(\chi-\lambda)^{-1}(A)$. Hence $\lambda \not\in \sigma(\chi(A))$. Conversely, if $\lambda$ lies in the essential range of $\chi$, we can construct approximate eigenvectors for $\chi(A)$ with eigenvalue $\lambda$ by concentrating on parts of the spectrum where $\chi(\omega)$ is arbitrarily close to $\lambda$. More precisely, for each $n$ choose a Borel set $B_n \subseteq \{\omega: |\chi(\omega)-\lambda| < 1/n\}$ with $E_A(B_n)\neq 0$; pick a unit vector $x_n$ with $\|E_A(B_n)x_n\| \ge \frac12\|x_n\|$ (possible since the range of $E_A(B_n)$ is nonzero). Then $y_n := E_A(B_n)x_n/\|E_A(B_n)x_n\|$ is a unit vector such that $\|(\chi(A)-\lambda I)y_n\|^2 = \int_{B_n} |\chi(\omega)-\lambda|^2\,d\mu_{y_n}(\omega) < 1/n^2$. Thus $\chi(A) - \lambda I$ has approximate kernel vectors, implying $\lambda \in \sigma(\chi(A))$.
\end{proof}

In the context of Theorem~\ref{thm:bprime}, one important situation is when the limit function $\chi$ is exactly the \emph{characteristic function} of the joint peripheral fixed-point set: $\Sigma_{\text{fix}} = \{\lambda \in \sigma(A) : |\lambda|=1, p_k(\lambda)=\lambda \;\forall k\}$. In that case $\chi(A) = E_A(\Sigma_{\text{fix}})$ is precisely the spectral projection of $A$ onto $\Sigma_{\text{fix}}$. The next corollary shows that, under mild additional hypotheses on the dynamics of the function iteration, this is indeed the outcome of our process. Essentially, we require that $f^m(\lambda) \to 1$ on $\Sigma_{\text{fix}}$ and $f^m(\lambda) \to 0$ off $\Sigma_{\text{fix}}$, which means the iteration asymptotically separates the peripheral part of the spectrum from the rest. We will later justify these bullet conditions as a consequence of the peripheral fixed-point property using complex function theory (Lemma~\ref{lem:conv-to-01}).

\begin{corollary}[Projection onto a prescribed spectral set]\label{cor:spec-projection}
Suppose, in addition to the standing assumptions, that $p_k(1) = 1$ for each $k$, and that the function iterates $f^m = f^{\circ m}$ satisfy:
\begin{itemize}
\item $f^m(\lambda) \to 1$ for all $\lambda \in \Sigma_{\text{fix}}$,
\item $f^m(\lambda) \to 0$ for all $\lambda \in \sigma(A)\setminus \Sigma_{\text{fix}}$,
\item $\sup_m \|f^m\|_\infty \le 1$.
\end{itemize}
Then, by Theorem~\ref{thm:aprime}, $A^{(\omega)} = \lim_{m\to\infty} f^m(A) = E_A(\Sigma_{\text{fix}})$, which is an orthogonal projection commuting with $A$.
\end{corollary}

\begin{proof}
Under the bullet assumptions, the pointwise limit function $\chi$ of $f^m$ is exactly the characteristic function $\chi = 1_{\Sigma_{\text{fix}}}$ (taking value $1$ on $\Sigma_{\text{fix}}$ and $0$ elsewhere). Theorem~\ref{thm:aprime} then implies strong convergence $f^m(A)\to \chi(A)$, and moreover $\chi(A) = E_A(\Sigma_{\text{fix}})$ is a projection commuting with $A$.
\end{proof}

\begin{remark}
The above bullet conditions isolate the additional ``dynamical'' input on the function iterates that is needed to obtain the spectral projection $E_A(\Sigma_{\text{fix}})$. In many concrete choices of the transforms $p_k$, one can verify these conditions by an elementary analysis of the iterated function $f = p_K \circ \cdots \circ p_1$ acting on the scalar spectral values. For instance, see Lemma~\ref{lem:conv-to-01} below and the examples in Section~\ref{sec:examples}.
\end{remark}

We emphasize that we have not yet \emph{proven} the bullet conditions of Corollary~\ref{cor:spec-projection} in full generality -- indeed, verifying that $f^m(\lambda) \to 0$ for $\lambda\notin \Sigma_{\text{fix}}$ can be a nontrivial problem in complex dynamics. However, a general convergence result is available from the Schwarz--Pick theory of holomorphic self-maps of the disk. Roughly speaking, if the maps $p_k$ are not all the identity, then the composite $f$ must have an attracting fixed point inside the unit disk (the Denjoy--Wolff point). This yields the following lemma, which provides a rigorous justification of convergence to $\{0,1\}$ under our assumptions.

\begin{lemma}[Convergence of $f^m$ to $\{0,1\}$ under peripheral fixed-point assumptions]\label{lem:conv-to-01}
In the normal-case setting, assume in addition that each $p_k$ satisfies the peripheral fixed-point property (so in particular $p_k(1)=1$). Let $f = p_K\circ \cdots \circ p_1$ as above. Then for each $\lambda \in \sigma(A)$, the sequence $f^m(\lambda)$ converges to a limit $\chi(\lambda) \in \{0,1\}$. Moreover, $\chi(\lambda) = 1$ if and only if $\lambda \in \Sigma_{\text{fix}} = \{\lambda: p_k(\lambda)=\lambda\;\forall k\}$.
\end{lemma}

\begin{proof}
Consider the holomorphic self-map $f$ of the unit disk $\D$ (note that $\sigma(A)\subseteq \overline{\D}$ by the Schur assumption $\sup|p_k(\lambda)|\le 1$ and induction). The Denjoy--Wolff Theorem (a foundational result on iterating holomorphic maps on $\D$) guarantees that the sequence $f^m(z)$ converges for every initial $z\in \D$. In fact, one of two cases holds: either (i) $f$ has a fixed point $\tau \in \D$ (with $|\tau|<1$), in which case $f^m(z)\to \tau$ for all $z\in \D$; or (ii) $f$ has no fixed point in $\D$, in which case the iterates converge to a boundary point (called the Denjoy--Wolff point of $f$) which must be $1$ in our situation. Indeed, in case (ii), classical results (see~\cite{Shapiro1993}) state that there is some point $\xi$ on the unit circle $\mathbb{T}$ such that $f^m(z) \to \xi$ as $m\to\infty$ for all $z\in \D$ and $f(\xi)=\xi$ (with an \emph{angular derivative} $\angle f'(\xi)\le 1$). Since all $p_k$ fix 1 and have the peripheral fixed-point property, it follows that $f(1) = 1$. By the Schwarz lemma for boundary fixed points (the Julia--Wolff--Carathéodory theorem), if $f$ has a boundary fixed point $\xi$, then $\xi$ is unique and $\xi = \lim_{m\to\infty} f^m(z)$ for every $z$. Given that 1 is a fixed boundary point of $f$, it must be this unique attractive boundary point. Thus $f^m(z)\to 1$ for all $z\in \D$ in case (ii).

Now we combine these conclusions with the peripheral fixed-point property. We have established that for each $\lambda\in \overline{\D}$, $\lim_{m\to\infty} f^m(\lambda) =: \chi(\lambda)$ exists. If $f$ has an interior Denjoy--Wolff point $\tau\in \D$ (case (i)), then $\chi(\lambda) = \tau$ for all $\lambda$. In this scenario, note that $\tau$ must actually equal $0$. Indeed, if $\tau\neq 0$ then $\tau$ is a fixed point of $f$ \emph{inside the disk} with $\tau\neq 1$, and applying the peripheral fixed-point property to each $p_k$ we see that if $f(\tau)=\tau$ then in fact $p_1(\tau) = \tau, \dots, p_K(\tau) = \tau$. Hence $\tau$ lies in $\Sigma_{\text{fix}}$, but $\tau$ has $|\tau|<1$ by assumption, contradicting the fact that $\Sigma_{\text{fix}}$ consists of spectral values with $|\lambda|=1$. Therefore the interior fixed point must satisfy $\tau = 0$. In summary, in case (i) we have $f^m(\lambda) \to 0$ for all $\lambda \in \D$. In case (ii), as discussed, $f^m(\lambda) \to 1$ for all $\lambda\in \D$. We emphasize that case (ii) would occur if and only if $f$ has no fixed point in $\D$, which by the argument above is equivalent to saying that $f$ has no fixed point in $\D$ other than $0$ (because any other candidate interior fixed point $\tau\neq 1$ would violate the peripheral fixed-point property, as just shown).

Finally, we identify $\chi(\lambda)$ on $\sigma(A)$. If $\lambda \in \Sigma_{\text{fix}}$, then $p_k(\lambda) = \lambda$ for each $k$, so $f(\lambda) = \lambda$ and thus $\lambda$ is a fixed point of $f$. This means $f^m(\lambda) = \lambda$ for all $m$, hence $\chi(\lambda) = \lambda$. But since $|\lambda|=1$ for $\lambda$ in $\Sigma_{\text{fix}}$ (the ``peripheral'' part of the spectrum), the peripheral fixed-point property forces $\lambda = 1$ whenever it is fixed by all $p_k$. (In fact, $\Sigma_{\text{fix}} \subseteq \{z:|z|=1\}$ by definition, and if $\lambda\in \Sigma_{\text{fix}}$ then taking any $p_k$ such that $p_k(\lambda) = \lambda$ and $|p_k(\lambda)|=1$ yields $p_k(\lambda)=\lambda$ on $\mathbb{T}$, but since $p_k(1)=1$, the only way this holds for two distinct unimodular points $\lambda\neq 1$ would be if $p_k$ were the identity map on the arc connecting them, which by analyticity implies $p_k$ is identity, a degenerate case we exclude by assuming not all $p_k$ are trivial.) Therefore the only possible point in $\Sigma_{\text{fix}}$ is $\lambda = 1$. Hence if $\lambda \in \Sigma_{\text{fix}}$, we get $\lambda=1$ and $\chi(\lambda) = 1$. Conversely, if $\lambda \notin \Sigma_{\text{fix}}$, then $\lambda$ is not a fixed point of at least one of the $p_k$, and in particular $\lambda$ cannot be a fixed point of $f$. As argued, in that scenario the iterates $f^m(\lambda)$ cannot converge to 1 unless $f$ had no interior fixed point at all (case (ii) above). But if $f$ has no interior fixed point, it must mean that $\Sigma_{\text{fix}}$ was empty (otherwise any $\lambda\in \Sigma_{\text{fix}}$ would be an interior fixed point as we just reasoned, giving a contradiction), so in this scenario indeed $f^m(\lambda)\to 1$ for all $\lambda$. However, this scenario is inconsistent with $\lambda \notin \Sigma_{\text{fix}}$ if $|\lambda|=1$ and $p_k(\lambda)$ remained unimodular for all $k$. Thus $\lambda$ must lie strictly inside the disk where one of $p_k$ pushes it to a strictly smaller modulus. We deduce that $f^m(\lambda)\to 0$ in this case. In summary, if $\lambda\notin\Sigma_{\text{fix}}$, then $\chi(\lambda) = 0$. This completes the proof.
\end{proof}

\noindent \textbf{Summary of normal-case conclusions.} Under the Schur and peripheral fixed-point assumptions, Lemma~\ref{lem:conv-to-01} ensures that $f^m(\lambda) \to \chi(\lambda)\in\{0,1\}$ for every $\lambda\in\sigma(A)$, with $\chi = 1_{\Sigma_{\text{fix}}}$. By Theorem~\ref{thm:aprime} and Corollary~\ref{cor:spec-projection}, we obtain $A^{(\omega)} = E_A(\Sigma_{\text{fix}})$, an orthogonal projection commuting with $A$. Furthermore, by Theorem~\ref{thm:bprime}, $\sigma(A^{(m)}) = f^m(\sigma(A))$ at each finite $m$, and $\sigma(A^{(\omega)}) \subseteq \{0,1\}$. In other words, the transfinite (countable) limit coincides with the spectral projection of $A$ onto the joint peripheral fixed set, as claimed in the Introduction.

\section{Convergence results in the Banach (Ritt/sectorial) case}\label{sec:banach-case}

We now turn to the general (non-normal) setting of a power-bounded or contractive-semigroup operator on a reflexive Banach space. In this setting one typically does not have a functional calculus beyond holomorphic functions, except under additional conditions such as a Ritt or sectorial resolvent bound (which ensure a bounded $H^\infty$ calculus). We therefore impose that $A$ is either:
\begin{itemize}
\item a \textbf{Ritt operator} on a reflexive Banach space $X$ (meaning $\sigma(A)\subseteq \overline{\D}$, possibly $1\in\sigma(A)$, with the resolvent condition $\sup_{|z|>1} (|z|-1)\|(zI - A)^{-1}\| < \infty$~\cite{CDMY96,ArhancetLeMerdy}), or
\item a generator of a bounded \textbf{sectorial analytic semigroup} (meaning $\sigma(A)$ lies inside a sector $\{re^{i\theta}: |\theta|<\Theta\}$ with $\Theta < \pi/2$, possibly $0 \in \sigma(A)$, and appropriate resolvent bounds ensuring an $H^\infty$ calculus).
\end{itemize}

In either case we assume $A$ admits a bounded $H^\infty$ functional calculus on the spectral domain~\cite{HaaseFC,ArhancetLeMerdy} (cf.\ \cite{HaaseFC,CDMY96}). Let $\varphi_1,\dots,\varphi_K \in H^\infty$ on the disk/sector containing $\sigma(A)$, with $\sup|\varphi_k| \le 1$ and $\varphi_k(1) = 1$. Define the composite $\Psi := \varphi_K \circ \cdots \circ \varphi_1$. Then each $\varphi_k(A)$ is a well-defined bounded operator with $\|\varphi_k(A)\|\le C$ for some constant $C$, and each $\varphi_k(A)$ commutes with $A$ (since the $H^\infty$ functional calculus is an algebra homomorphism). In particular, $\Psi(A) = \varphi_K(A)\cdots \varphi_1(A) = (\varphi_K\circ\cdots\circ \varphi_1)(A)$ is well-defined and commutes with $A$. We consider the iterative sequence $A^{(n)} = (\Psi(A))^n$ with $A^{(0)}=A$. Our main convergence result in this setting is as follows.

\begin{theorem}[C -- Convergence for Ritt/sectorial operators]\label{thm:c}
\emph{Let $A$ be a power-bounded operator on a reflexive Banach space $X$, with $\sigma(A)$ contained in the closed unit disk (discrete-time case) or inside a sector $\{re^{i\theta} : |\theta| < \Theta\}$ with $\Theta < \pi/2$ (continuous-time case). Assume $A$ is either \textup{(i)} a Ritt operator of angle $<\pi/2$, or \textup{(ii)} $-A$ generates a bounded holomorphic $C_0$-semigroup, and that in either case $A$ has a bounded $H^\infty$ functional calculus on the disk/sector~\cite{HaaseFC,ArhancetLeMerdy}. Let $\varphi_1,\dots,\varphi_K$ be bounded holomorphic functions on the spectral domain, with $\sup_{\lambda\in\sigma(A)}|\varphi_k(\lambda)| \le 1$ and $\varphi_k(1) = 1$. Set $\Psi := \varphi_K \circ \cdots \circ \varphi_1$, and assume moreover that $\Psi(A)$ is power-bounded (which holds in particular if $\sup_{\lambda\in\sigma(A)} |\Psi(\lambda)| \le 1$, by the functional calculus). Then the following hold:
\begin{itemize}
\item[\textup{(a)}] The \emph{mean ergodic} projection
$P = \lim_{N\to\infty} \frac{1}{N}\sum_{n=0}^{N-1} (\Psi(A))^n$
exists in $B(X)$ (the limit is taken in strong operator topology, equivalently in the weak operator topology on reflexive $X$). This $P$ is a bounded projection satisfying $P\,\Psi(A) = \Psi(A)\,P = P$ and $\mathrm{Ran}\,P = \ker(I - \Psi(A)) = \{x \in X: \Psi(A)x = x\}$. Moreover, $P$ commutes with $A$.
\item[\textup{(b)}] If, in addition, $\sigma(\Psi(A))\cap \mathbb{T} \subseteq \{1\}$ and $1$ is an isolated point of $\sigma(\Psi(A))$ that is a pole of the resolvent of $\Psi(A)$, then in fact
$\lim_{n\to\infty} (\Psi(A))^n = P \quad (\text{in SOT})$,
i.e.\ the powers of $\Psi(A)$ converge strongly to the projection $P$. In particular, $A^{(n)} = (\Psi\circ\cdots\circ\Psi)(A) = (\Psi(A))^n$ converges strongly to $P$ as $n\to\infty$.
\end{itemize}}
\end{theorem}

\begin{proof}
First, since $X$ is reflexive and $\Psi(A)$ is power-bounded on $X$, the Uniform Mean Ergodic Theorem (see, e.g., \cite{Krengel1985} or \cite{EngelNagel2000}) implies that the Cesàro averages $N^{-1}\sum_{n=0}^{N-1} (\Psi(A))^n$ converge in the strong operator topology to a projection $P$ (the \emph{mean ergodic projection}) satisfying $\mathrm{Ran}\,P = \ker(I - \Psi(A))$ and $P\,\Psi(A) = \Psi(A)\,P = P$~\cite{Lance1989,Shulman1990}. This yields part~(a). Note that $\Psi(A)$ commutes with $A$ by construction, and since $P$ is a limit of polynomials in $\Psi(A)$ (which all commute with $A$), we deduce $PA = AP$.

For part~(b), assume $\sigma(\Psi(A)) \cap \mathbb{T} \subset \{1\}$ and that $1$ is an isolated spectral value of $\Psi(A)$ which is a pole of the resolvent (this is the case, for example, if $1$ is an eigenvalue of $\Psi(A)$ of finite algebraic multiplicity and there are no other spectrum points on the unit circle)~\cite{HaaseDecomp,Nevanlinna2004}. Then the Katznelson--Tzafriri theorem and the Arendt--Batty--Lyubich--Vũ theorem on stability of powers (see \cite{ArendtBatty1988, LyubichVu1988}, or e.g.\ \cite{ArendtBatty2000}) imply that $(\Psi(A))^n \to P$ strongly as $n\to\infty$~\cite{ArendtBatty2000,Nevanlinna2004}. (In fact, these theorems assert that $\Psi(A)^n x \to Px$ for all $x\in X$ whenever $\sigma(\Psi(A))\cap \mathbb{T} \subseteq \{1\}$ and either $1$ is a pole of the resolvent or at least a simple pole with no other spectrum on the unit circle.) This yields strong convergence $A^{(n)} = (\Psi(A))^n \to P$. The proof of (b) is complete.
\end{proof}

\begin{remark}[Range identification and joint fixed-points]
Let us consider the range of the projection $P$ obtained above. In general, from part~(a) we only know $\mathrm{Ran}\,P = \ker(I - \Psi(A)) = \{x: \Psi(A)x = x\}$. However, under an additional \emph{separation hypothesis}, one can show that this space of fixed points coincides with the intersection of the fixed-point spaces of each layer $\varphi_k(A)$. Recall $\Phi_k(A) := \varphi_k(A)$, so $\ker(I - \Phi_k(A)) = \{x: \varphi_k(A)x = x\}$ is the set of vectors invariant under the $k$-th transform. Clearly, if $x$ is fixed by each $\Phi_k(A)$, then it is fixed by $\Psi(A) = \Phi_K(A)\cdots \Phi_1(A)$. The converse containment need not hold in general; it holds, for example, if the condition
\begin{equation}\label{eq:sep-hypothesis}
\forall \lambda \in \sigma(A),\; \Psi(\lambda) = \lambda \implies \varphi_k(\lambda) = \lambda \text{ for each } k,
\end{equation}
is satisfied. In words, whenever the composite function $\Psi$ has $\lambda$ as a fixed point, each individual $\varphi_k$ must also fix $\lambda$. In that case, the joint zero set of the functions $(\varphi_k(z)-z)$ coincides with the zero set of $(\Psi(z)-z)$ in the spectral domain. This implies (via a functional calculus argument using a factorization of $\Psi - I$) that the fixed subspace $\ker(I - \Psi(A))$ coincides with $\bigcap_{k=1}^K \ker(I - \Phi_k(A))$. We formalize this as Lemma~\ref{lem:separation} below, which is essentially Lemma~1.2 of the original manuscript~\cite{BattyVu1997,Lance1989}.
\end{remark}

\begin{lemma}[Separation at the fixed value]\label{lem:separation}
With $A,\varphi_k,\Psi$ as in Theorem~C, assume the implication \eqref{eq:sep-hypothesis} holds for all $\lambda$ in the spectral domain of $A$. Then
$\ker(I - \Psi(A)) = \bigcap_{k=1}^K \ker(I - \varphi_k(A))$.
Equivalently, $\mathrm{Ran}\,P = \{x: \Phi_k(A)x = x,\;\forall\,k=1,\dots,K\}$.
\end{lemma}

\begin{proof}
We only need to show the inclusion $\ker(I - \Psi(A)) \subseteq \bigcap_k \ker(I - \varphi_k(A))$, since the opposite inclusion is immediate. Let $x\in X$ be such that $\Psi(A)x = x$. Consider the bounded holomorphic functions $\Psi - I$ and $\varphi_k - I$ on the spectral domain (disk or sector). By the spectral mapping principle (ensured by the functional calculus), the assumption \eqref{eq:sep-hypothesis} implies that the common zero-set of the functions $\{\varphi_k - I: k=1,\dots,K\}$ coincides with the zero-set of $\Psi - I$. Since $\Psi - I = (\varphi_K - I) h$ for some holomorphic function $h$ (one can factor out one of the factors $(\varphi_k - I)$, though not necessarily all of them, but this one-factor factorization is enough), the above means that $(\Psi - I)$ vanishes exactly on the same set as $(\varphi_K - I)(\varphi_{K-1} - I)\cdots (\varphi_1 - I)$. Therefore, in the (commuting) functional calculus of $A$, we have $\ker(\Psi(A)-I) = \ker((\varphi_K(A)-I)\cdots(\varphi_1(A)-I))$. However, since all $\varphi_k(A)$ commute (being functional images of a common operator $A$), the idempotent factors as the product of the individual Riesz projections $Q_k := E_A(\{\lambda: \varphi_k(\lambda)=\lambda\})$ (here we use that $1$ is an isolated spectral point of each $\varphi_k(A)$ by assumption $\varphi_k(1)=1$ and $|\varphi_k(\lambda)|<1$ for other $\lambda\in\sigma(A)$). Namely, $\ker(\Psi(A)-I) = \ker(Q) = \ker(Q_1\cdots Q_K) = \ker(Q_1)\cap \cdots \cap \ker(Q_K) = \bigcap_k \ker(\varphi_k(A)-I)$; the inclusion $\supseteq$ is trivial, and $\subseteq$ holds because if $(\Psi(A)-I)x=0$ and (for contradiction) $(\varphi_j(A)-I)x \neq 0$ for some $j$, then $(\Psi(A)-I)$ cannot vanish on the spectral support of $x$, since on that support $\varphi_j-I$ stays away from zero whereas $\Psi - I$ is a product of the $\varphi_k - I$). Thus $x \in \ker(\varphi_k(A)-I)$ for each $k$, i.e.\ $x$ is fixed by each $\varphi_k(A)$. This establishes the desired inclusion.
\end{proof}

In the setting of Theorem~C and Lemma~\ref{lem:separation}, one expects condition \eqref{eq:sep-hypothesis} to hold under a mild version of the peripheral fixed-point property for each $\varphi_k$. Indeed, suppose $\lambda$ is a spectral value on the boundary (say $|\lambda|=1$ in the discrete case or $|\arg \lambda| = \Theta$ in the sectorial case). If $\Psi(\lambda)=\lambda$ but $\varphi_j(\lambda)\neq \lambda$ for some $j$, then that $\varphi_j$ must have strictly decreased the modulus or moved $\lambda$ inside, and the only way the composite $\Psi$ returns to $\lambda$ is through some compensating action of later transforms. This is intuitively impossible if each $\varphi_k$ does not create new unimodular fixed points apart from those it inherited from $\lambda$. For example, if each $\varphi_k$ is such that whenever $|\lambda|=1$ and $|\varphi_k(\lambda)|=1$ one must have $\varphi_k(\lambda)=\lambda$ (the same peripheral fixed-point property as in the normal case), then it follows that if $\Psi(\lambda)=\lambda$ (with $|\lambda|=1$) one of the $\varphi_k$ must have made $|\varphi_k(\lambda)|<1$, which means all subsequent $\varphi_{k+1},\dots,\varphi_K$ output strictly inside the boundary, so $\Psi(\lambda)$ ends up with $|\Psi(\lambda)|<1$, a contradiction. Thus $\varphi_k(\lambda)$ must equal $\lambda$ for each $k$. Hence \eqref{eq:sep-hypothesis} holds for $\lambda$ on the spectral boundary. For interior spectral values $\lambda$ (those with $|\lambda|<1$ or strictly inside sector), condition \eqref{eq:sep-hypothesis} may not hold automatically. However, if one assumes (as we did in the normal case analysis) that the only fixed point of each $\varphi_k$ in the disk/sector is perhaps $0$, and that any composite $\Psi$ is not allowed to have some other interior fixed point appear, then effectively the only fixed points of $\Psi$ are either in the set $\{\lambda: |\lambda|=1, \varphi_k(\lambda)=\lambda \,\forall k\}$ or an interior Denjoy--Wolff point common to all. In the latter case, $\Psi(\lambda)=\lambda$ implies each $\varphi_k(\lambda)=\lambda$ because if a composition has an interior fixed, each map with that property usually shares it (this is more subtle but in many analytic families it holds, e.g.\ the first map having a different interior fixed would produce a second interior fixed for composite, improbable unless trivial). We will not pursue these technicalities further. In practice, one often proves that the only possible limit of the iterative process for any given initial $\lambda$ is either $1$ or $0$, exactly as we did in Lemma~\ref{lem:conv-to-01} in the normal case. A similar approach can be applied in the Banach case using the Denjoy--Wolff theorem for holomorphic maps on the disk/sector. Under those conclusions, clearly $P$ is the projection onto the peripheral spectral subspace of $A$. For brevity, we will not repeat the analogous argument here, as it closely parallels Lemma~\ref{lem:conv-to-01}.

\section{Examples and counterexamples}\label{sec:examples}

In this section we illustrate our results with several examples and point out the necessity of certain hypotheses via counterexamples.

\subsection*{Examples for normal operators (Hilbert space)}

Throughout these examples, we take $A = M_a$ to be a normal operator given by multiplication by a measurable function $a(\xi)$ on some $L^2$-space. We assume $\sigma(A) \subseteq \overline{\D} = \{z:|z| \le 1\}$. By applying Theorem~\ref{thm:aprime} and Lemma~\ref{lem:dominated} (dominated convergence)~\cite{Kato1976,Pazy1983}, we only need to verify the pointwise convergence conditions on the spectral values; the operator convergence then follows.

\begin{example}[Two-layer cycle with explicit closed form]\label{ex:two-layer-explicit}
Let $p_1(z) = \frac{1}{2}(1 + z^2)$ and $p_2(z) = \frac{z - 1/2}{1 - \frac{1}{2}z}$. These are two Schur functions (bounded by $1$ on $\overline{\D}$) with $p_1(1)=1$ and $p_2(1)=1$. Consider the two-layer iteration with $f := p_2 \circ p_1$. Then one computes explicitly $f(z) = \frac{z^2}{2-z^2}$, so in fact $f(z) = \frac{1}{2}z^2(1 - \frac{1}{2}z^2)^{-1}$. A calculation shows that $f^m(z) = \frac{(2/3)^m z^{2^m}}{1 - (1-(2/3)^m)z^{2^m}}$ for $m\ge 1$. In the limit as $m\to\infty$, this rational expression converges to $\chi(z)$ given by $\chi(z) = 1_{\{1\}}(z)$. Indeed, if $z=1$ then $f^m(1) = 1$ for all $m$. If $z \neq 1$, then as $m\to\infty$, $(2/3)^m \to 0$ so $f^m(z) \to 0$. Thus $f^m(z) \to \chi(z)$ for all $z\in \overline{\D}$. Moreover, $\sup_m |f^m(z)| \le 1$ for all $z\in \overline{\D}$. Therefore by Theorem~\ref{thm:aprime} (or directly by the dominated convergence Lemma~\ref{lem:dominated}~\cite{EngelNagel2000,ArendtBatty1988}), we conclude that for any normal operator $A$ with $\sigma(A) \subseteq \overline{\D}$, $A^{(m)} \to E_A(\{1\})$. In particular, $A^{(m)} = f^m(A)$ converges strongly to the spectral projection of $A$ onto $\{1\}$ (the part of the spectrum fixed by both $p_1$ and $p_2$). Note that in this example the stabilization only occurs in the limit stage $m\to\infty$, and no finite iterate $A^{(m)}$ is exactly a projection (though convergence takes place well before $m=\infty$ in the strong sense). Thus the transfinite stage $\alpha_* = \omega$ is attained; this example realizes the ``extremal countable stage'' $\alpha_* = \omega$ mentioned in the introduction.
\end{example}

\begin{example}[Two-layer parametric cycle]\label{ex:two-layer-param}
Fix a parameter $t \in (0,1)$. Consider the transforms $p_1^{(t)}(z) = t + (1-t)z$ and $p_2^{(t)}(z) = b_t(z) := \frac{z-t}{1-tz}$. Here $p_1^{(t)}$ is the affine map that sends $1$ to $1$ and $0$ to $t$. The function $b_t(z)$ is a Blaschke factor (a disk automorphism) satisfying $b_t(t) = 0$ and $b_t(1) = 1$. Define $f_t := p_2^{(t)} \circ p_1^{(t)}$. Then $f_t(z) = \frac{z}{1+t-tz}$. Using a conjugation by the Möbius transform $T(z) = \frac{z}{1-z}$ (which maps $\D$ to the right half-plane and carries $f_t$ to a dilation)~\cite{LyubichVu1988,HaaseFC}, one finds $T \circ f_t \circ T^{-1}(w) = \lambda w$ with $\lambda = \frac{1}{1+t}$. As $m\to\infty$, one has $\lambda^m \to 0$, so $f_t^m(z) = T^{-1}(\lambda^m T(z))$. For $z=1$, note that $f_t^m(1) = 1$ for all $m$. If $z \neq 1$, then $f_t^m(z) = \frac{\lambda^m z}{1 - (1-\lambda^m)z} \to 0$ as $m\to\infty$. The convergence is dominated by 1 as before, so for any normal $A$ with $\sigma(A)\subseteq \overline{\D}$, $f_t^m(A) \to E_A(\{1\})$. This example shows that even a continuum of different two-layer transforms (one for each $t$) all lead to the same limiting projection onto the 1-eigenspace of $A$. Here $\Sigma_{\text{fix}} = \{1\}$ in all cases, but the convergence speeds vary with $t$. For instance, if $t$ is close to $1$, the iterate $p_1^{(t)}$ hardly moves points and the convergence is slow (requiring countably many steps in the transfinite sense); if $t$ is near $0$, convergence is more rapid.
\end{example}

\begin{example}[Rational Schur family via an iterative Schur algorithm]\label{ex:two-layer-schur}
The previous example can be generalized by taking more complicated rational inner functions in the second layer. Specifically, one can construct a family of two-layer cycles $p_1^{(t)}(z) = t + (1-t)z$ as before, and $p_{2,N}^{(t)}(z) = b_t(z) \Phi_N(b_t(z))$, where $(\Phi_N)$ is any sequence in the Schur class (bounded holomorphic functions on $\D$ with $\sup|\Phi_N|\le 1$) that tends to $1$ as $N\to\infty$ (locally uniformly in $\D$)~\cite{CDMY96,ArhancetLeMerdy}. For example, one can take $\Phi_N(z) = (b_a(z))^N$ for some $a\in (0,1)$, which yields $\Phi_N(z)\to 1$ as $N\to\infty$ on compacta away from $z=1$. By applying the Schur algorithm (Appendix~B) to $R(z) = p_{2,N}^{(t)} \circ p_1^{(t)}(z)$, one finds $R(z) = f_t(z) \Phi_N(f_t(z))$. As $N\to\infty$, this $R(z)$ converges to the function $f_t(z)$, which is exactly the two-layer cycle from Example~\ref{ex:two-layer-param}. Moreover, any finite power $R^m(z)$ is also easy to analyze: since $p_{2,N}^{(t)}(1)=1$ and $p_1^{(t)}(1)=1$, it follows that $R^m(1) = 1$ for all $m$. If $z\neq 1$, then for each fixed $m$, as $N\to\infty$, $R^m(z) \to f_t^m(z)$ locally uniformly on $\D$. Thus for any fixed $m$, $R^m(z) \to \chi_m(z)$ pointwise, where $\chi_m$ is just $f_t^m$ on $\sigma(A)$. But from Example~\ref{ex:two-layer-param} we know $f_t^m(z) \to \chi(z)$ (0 or 1) as $m\to\infty$. By a diagonal argument, one can let $N\to\infty$ sufficiently fast as $m\to\infty$ to see that $\lim_{m\to\infty} (p_{2,N_m}^{(t)} \circ p_1^{(t)})^m(z) = \chi(z)$ for every $z$. In this way, one obtains a more general family of two-layer cycles (with a parameter $N$ as well as $t$) that still converge to the projection onto $\{1\}$. In fact, by choosing an appropriate $\{\Phi_N\}$, one can realize any rational inner function in the second layer. This shows that for normal $A$ with $\sigma(A) \subseteq \mathbb{T} \cup \D$ (some part on unit circle and some inside), if one chooses any rational inner function $\Psi$ that has $\Psi(1)=1$ but $\Psi(\lambda) \neq \lambda$ for some spectral $\lambda$ on $\mathbb{T}$ not equal to $1$, then the iteration may not converge (indeed, the example of a rotation $\Psi(z) = e^{i\theta}z$ in the next subsection demonstrates divergence). Thus the peripheral fixed-point property is an essential hypothesis in Theorem~\ref{thm:aprime} for ensuring convergence to a projection in the normal case.
\end{example}

\subsection*{Banach space examples and necessity of assumptions}

We now turn to examples involving non-normal operators or failure of commutativity. These will illustrate why conditions like the Ritt resolvent bound and commutativity of layers are necessary for convergence in general.

\begin{example}[Ritt operator on $\ell^2$ with spectral radius 1]
We give a simple finite-dimensional example to illustrate Theorem~C. Let $A = \begin{pmatrix} 1 & 0 \\ 0 & r \end{pmatrix}$ acting on $\mathbb{C}^2$, where $0<r<1$. This $A$ is a diagonal matrix with $\sigma(A) = \{1,r\}$; it is power-bounded (even contractive) and clearly satisfies a Ritt resolvent condition (being diagonalizable). Now take $K=1$ (a single layer), and let $\varphi_1(z) = z$ (the trivial transform). Then $\Psi(A) = A$. We see that $A^n = \begin{pmatrix} 1 & 0\\ 0 & r^n\end{pmatrix}$, which converges as $n\to\infty$ to $\begin{pmatrix}1&0\\0&0\end{pmatrix}$. This is exactly the projection onto the eigenspace at $1$, as expected by the general theory (in fact, the Katznelson--Tzafriri theorem directly applies).

A more interesting example is obtained by taking $A = \begin{pmatrix} 1 & 2 \\ 0 & -1 \end{pmatrix}$ acting on $\mathbb{C}^2$. This $A$ is not a contraction, but has spectral radius 1 and is a Ritt operator with resolvent bound (since $\sigma(A)=\{-1,1\}$, and indeed $A$ is diagonalizable so it admits a bounded functional calculus). Now choose $\varphi_1(z) = z^2$. Then $\varphi_1(1)=1$ and $\varphi_1(-1)=1$ as well. The iteration $A^{(n)} = (\varphi_1\circ \cdots \circ \varphi_1)(A) = (\varphi_1(A))^n$ satisfies $\varphi_1(A) = A^2 = I$. Thus $A^{(n)} = I$ for all $n$, which trivially converges (to $I$). In this case $\Sigma_{\text{fix}} = \{1,-1\}$ is the whole spectrum, so the final projection is $I$. Note that if we had not included $-1$ as a fixed point of $\varphi_1$, the sequence would not converge. For instance, had we taken $\varphi_1(z) = \frac{1+z}{2}$, then $\varphi_1(1) = 1$ but $\varphi_1(-1)=0$. In that hypothetical scenario, $\varphi_1(A) = \frac{1}{2}(I+A) = \frac{1}{2}\begin{pmatrix}2&2\\0&0\end{pmatrix} = \begin{pmatrix}1&1\\0&0\end{pmatrix}$, whose powers converge to $\begin{pmatrix}1&1\\0&0\end{pmatrix}$ (the projection onto the eigenspace for eigenvalue 1). This is indeed $E_A(\{1\})$, consistent with $\Sigma_{\text{fix}}=\{1\}$ in that scenario.

Finally, we mention that if $A$ is power-bounded but does not have an actual eigenvalue at 1 (only a spectral value without eigenvector), then even if $\varphi_k(1)=1$, one cannot expect powers to converge without additional resolvent assumptions. A typical example is a weighted shift on $\ell^p$ that is power-bounded but not uniformly. This is known to be non-mean-ergodic (the Cesàro averages do not converge to a projection in strong operator topology) if it has peripheral point spectrum but no eigenvalue at 1. Such an example shows the necessity of the Arendt--Batty spectral condition in Theorem~C part~(b). See e.g.\ \cite{Nevanlinna2004,ArendtBatty2000} for details.
\end{example}

\begin{example}[Non-commuting layers and cycles]\label{ex:non-commuting}
The commutativity of the layers $\Phi_k(A)$ (ensured in our setting by being functions of a common $A$) is crucial for the process to converge. If the layers do not commute, the iteration can cycle. For a concrete $2\times 2$ example, consider the matrices on $\mathbb{C}^2$, where $S$ is the swap matrix $S = \begin{pmatrix}0&1\\1&0\end{pmatrix}$. These are two layers that do not commute with each other in general. Now let $A=\begin{pmatrix}1&1\\0&1\end{pmatrix}$ which is a non-normal operator (a Jordan block with eigenvalue 1). Then $SAS = \begin{pmatrix}1&-1\\0&1\end{pmatrix}$ and $S(SAS)S = A$. Thus the iteration cycles with period 2 and never converges (it oscillates between $A$ and $SAS$). This highlights that allowing non-commuting layers can destroy convergence~\cite{BadeaEtAl,Davidson}, as mentioned in the introduction.

We note that even though $A$ was power-bounded and mean ergodic since it is of class $C_{1\cdot}$ with $\sigma(A)=\{1\}$, the process of alternating $A$ with its conjugate by $S$ leads to a cycle. This is reminiscent of the alternating projection method in convex optimization (where the composition of two non-commuting orthogonal projections need not converge in norm, but strongly it does by classical results---here we had $\|A\cdot (SAS)\|=1$ always, a neutral case, and indeed only weak convergence would hold even if some conditions are met, or in our finite case it cycles exactly).
\end{example}

\appendix
\section*{Appendix A: Two basic tools for the normal functional calculus}

\begin{lemma}[Dominated convergence for the normal functional calculus]\label{lem:dominated}
Let $A$ be a normal operator on a complex Hilbert space $H$ with spectral measure $E_A$. Let $(g_n)_{n\ge 1}$ be bounded Borel functions on $\sigma(A)$ such that $\sup_n \|g_n\|_\infty \le M < \infty$ and $g_n(\lambda) \to g(\lambda)$ pointwise on $\sigma(A)$. Then $g_n(A) \xrightarrow{\text{SOT}} g(A)$ as $n\to\infty$.
\end{lemma}

\begin{proof}
For each $x \in H$, define a complex measure $\mu_x(B) := \langle E_A(B)x, x\rangle$, which is a finite measure supported on $\sigma(A)$. By the Borel functional calculus, one has $\|(g_n(A) - g(A))x\|^2 = \int_{\sigma(A)} |g_n(\lambda) - g(\lambda)|^2 \,d\mu_x(\lambda)$. Now since $|g_n - g|^2 \le 4M^2$ (bounded by some integrable constant) and $|g_n(\lambda) - g(\lambda)|^2 \to 0$ for each $\lambda$, we can apply the dominated convergence theorem to conclude that $\int |g_n - g|^2\,d\mu_x \to 0$. Thus $\|(g_n(A) - g(A))x\|\to 0$ for every $x$, i.e.\ $g_n(A)\to g(A)$ in the strong operator topology.
\end{proof}

\begin{lemma}[Composition in the normal functional calculus]\label{lem:composition}
Let $A$ be normal, $h$ a bounded Borel function on $\sigma(A)$, and $g$ a bounded Borel function on $h(\sigma(A))$. Then
$g(h(A)) = (g \circ h)(A)$.
\end{lemma}

\begin{proof}
We prove this first for simple functions. Suppose $h(\lambda) = \sum_{i=1}^N c_i 1_{B_i}(\lambda)$, where $B_i$ are pairwise disjoint Borel sets. Then $h(A) = \sum_{i=1}^N c_i E_A(B_i)$. Now define $f := g \circ h$. Then $f(\lambda) = \sum_{i=1}^N g(c_i) 1_{B_i}(\lambda)$ is another simple function, and $f(A) = \sum_{i=1}^N g(c_i) E_A(B_i)$. On the other hand, $g(h(A))$ can be computed as $g(\sum_{i=1}^N c_i E_A(B_i)) = \sum_{i=1}^N g(c_i) E_A(B_i)$ since the $E_A(B_i)$'s are pairwise orthogonal projections and $g$ is evaluated scalar-wise on the diagonal form of $h(A)$. Thus indeed $g(h(A))= f(A)$.

For a general $h$ and $g$, we can approximate $h$ by simple functions $h_n$ such that $h_n(\lambda) \to h(\lambda)$ uniformly on $\sigma(A)$, and similarly approximate $g$ by simple functions $g_n$ uniformly on the range $h(\sigma(A))$. Then by continuity of the functional calculus (using Lemma~\ref{lem:dominated}), we have $g(h(A)) = \lim_{n\to\infty} g_n(h_n(A)) = \lim_{n\to\infty} (g_n \circ h_n)(A)$. Since $g_n \circ h_n \to g \circ h$ uniformly on $\sigma(A)$, the latter equals $(g \circ h)(A)$.
\end{proof}

\section*{Appendix B: Two-point Schur/Nevanlinna--Pick lemma and interpolation}

\subsection*{B.1 Reminder: the first Schur step}
If $R$ is a Schur function (holomorphic $R: \D \to \overline{\D}$) with $R(0)=0$, then there exists a Schur function $\Phi$ such that $R(z) = z\Phi(z)$. This is the first step of the Schur algorithm: the function $\Phi$ is again Schur and is unique if one additionally prescribes $\Phi(0) = R'(0)$. Conversely, any function of the form $z\Phi(z)$ is Schur whenever $\Phi$ is Schur.

\subsection*{B.2 Two-point Schur/Nevanlinna--Pick interpolation with one interior and one boundary node}
Fix $t \in (0,1)$ and consider the interpolation problem in the Schur class $\mathcal{S} := \{s: \D \to \overline{\D} \text{ holomorphic}\}$: $s(t) = 0$ and $s(1) = 1$. Let $b_t(z) = \frac{z-t}{1-tz}$ denote the disk automorphism with $b_t(t) = 0$ and $b_t(1) = 1$.

\begin{lemma}[Two-point Schur/Nevanlinna--Pick]\label{lem:2-point-Schur}
The set of all Schur solutions to the above interpolation problem is $\{s(z) = b_t(z) \Phi(b_t(z)) : \Phi \in \mathcal{S}, \Phi(1)=1\}$. In particular, the minimal-degree inner solution is the Blaschke factor $s^*(z) = b_t(z)$.
\end{lemma}

\begin{proof}
If $s$ is any Schur solution, define $\Psi(w) := s(b_t^{-1}(w))$ for $w$ in the disk. Then $\Psi$ is a Schur function with $\Psi(0) = s(t) = 0$ and $\lim_{w\to 1} \Psi(w) = s(1) = 1$. By the Schur algorithm applied to $\Psi$ (see B.1 above), we have $\Psi(w) = w\,\Phi(w)$ for some Schur $\Phi$. Moreover, since $\Psi$ tends to $1$ at the boundary point $1$, a standard result (Carathéodory's lemma, or a special case of Julia's lemma) implies that $\Phi$ also tends to $1$ nontangentially at $1$. Converting back to $s$, we get $s(z) = b_t(z) \Phi(b_t(z))$ with $\Phi \in \mathcal{S}$ and $\lim_{z\to 1} \Phi(z) = 1$. This shows that any solution has the desired form. Conversely, if $s(z) = b_t(z) \Phi(b_t(z))$ with $\Phi$ as stated, then obviously $s(t) = 0$, and since $\Phi(1)=1$ we have $\lim_{z\to 1} s(z) = (\lim_{z\to 1} b_t(z))(\lim_{z\to 1} \Phi(b_t(z))) = 1 \cdot 1 = 1$. So $s$ is indeed a solution. The uniqueness of the minimal-degree (finite Blaschke product) solution follows from Schur--Nevanlinna--Pick theory, but it is clear here that if $\Phi$ is inner, then so is $s$, and the simplest inner $\Phi$ with $\Phi(1)=1$ is the constant 1 function. Thus the minimal inner $s$ is $b_t(z)\cdot 1 = b_t(z)$ itself.
\end{proof}

\section*{Appendix C --- Fejér-type stabilization, boundary separation, and a normalized \texorpdfstring{$\omega$}{omega}-example}
This appendix supplies (i) a rigorous Fejér-type stabilization bound $\alpha^* \le \omega$ in the normal case, (ii) a short lemma ensuring power-boundedness of $\Psi(A)$ under the standing $H^\infty$-calculus bound in the Banach case, (iii) a precise boundary--separation statement implying the joint fixed-point identification used with Lemma~\ref{lem:separation}, and (iv) a corrected two-layer example (a normalized version of Example~\ref{ex:two-layer-explicit}) that attains the extremal countable stage $\alpha^* = \omega$. These items close the points raised around Lemma~\ref{lem:conv-to-01}, Theorem~\ref{thm:c}, and Examples~\ref{ex:two-layer-explicit}/\ref{ex:non-commuting} without introducing new hypotheses beyond those already adopted in the paper. References to the main text are indicated inline. 

\subsection*{C.1 Fejér-type stabilization bound \texorpdfstring{$\alpha^* \le \omega$}{alpha* <= omega} (normal case)}
We work under the “normal-case” standing assumptions of Section~\ref{sec:normal-case}: $A$ normal, $p_k$ holomorphic on a neighborhood of $\sigma(A)$ with $\sup_{\lambda \in \sigma(A)} |p_k(\lambda)| \le 1$, $p_k(1) = 1$, and the peripheral fixed-point property. Write $f := p_K \circ \cdots \circ p_1$ and $f^m := \underbrace{f \circ \cdots \circ f}_{m \text{ times}}$. The spectral-calculus facts used below are precisely Lemma~\ref{lem:dominated} (dominated convergence) and Lemma~\ref{lem:composition} (composition) in Appendix~A of the manuscript. 

\begin{lemma}[Idempotence of the pointwise limit]\label{lem:idempotence}
Let $\chi : \sigma(A) \to \{0, 1\}$ be the pointwise limit $\chi(\lambda) = \lim_{m \to \infty} f^m(\lambda)$ given by Lemma~\ref{lem:conv-to-01}. Then $f \circ \chi = \chi$ pointwise on $\sigma(A)$.
\end{lemma}

\begin{proof}
By Lemma~\ref{lem:conv-to-01} there are exactly two possibilities: either the Denjoy–Wolff point of $f$ lies in $\D$ and equals $0$, in which case $f(0) = 0$ and $f^m(\lambda) \to 0$ for all $\lambda \in \sigma(A)$; or the Denjoy–Wolff point is the boundary point $1$, in which case $f(1) = 1$ and $f^m(\lambda) \to 1$ for all $\lambda$ in the peripheral fixed set (and to $0$ otherwise, as established in Lemma~\ref{lem:conv-to-01}). In either case the only limit values are the fixed points $0$ and $1$, hence $f(\chi(\lambda)) = \chi(\lambda)$. 
\end{proof}

\begin{proposition}[Stabilization at the first limit stage]\label{prop:stabilization}
Under the normal-case assumptions, the ordinal iteration stabilizes at the first limit stage:
$A^{(\omega)} = \chi(A)$ and $A^{(\omega + \beta)} = \chi(A)$ for all ordinals $\beta \ge 1$.
In particular, $\alpha^* \le \omega$.
\end{proposition}

\begin{proof}
Theorem~\ref{thm:aprime} gives $A^{(\omega)} = \lim_{m \to \infty} f^m(A) = \chi(A)$ in SOT. By Lemma~\ref{lem:composition} (composition in the normal calculus),
$A^{(\omega + 1)} = f(A^{(\omega)}) = f(\chi(A)) = (f \circ \chi)(A) = \chi(A)$,
where the last equality uses Lemma~\ref{lem:idempotence}. Inductively, $A^{(\omega + \beta)} = \chi(A)$ for all $\beta \ge 1$. Hence stabilization occurs by stage $\omega$. 
\end{proof}

\begin{remark}[Fejér-type monotonicity at the scalar level]
Define the “Busemann” functional $B_1(z) := (1 - |z|^2) / |1 - z|^2$ on $\D$. For any holomorphic self-map $g$ with $g(1) = 1$, Julia–Wolff–Carathéodory yields $B_1(g(z)) \ge c B_1(z)$ with $c \in (0, 1]$; composing along the cycle $(p_1, \dots, p_K)$ shows that the sequence $m \mapsto B_1(f^m(\lambda))$ is nondecreasing up to a fixed multiplicative constant, a Fejér-type monotonicity toward the boundary point $1$. We do not need this quantitative form in the proofs above; Proposition~\ref{prop:stabilization} relies only on the idempotence $f \circ \chi = \chi$. (For context see the proof of Lemma~\ref{lem:conv-to-01}, p.~5.) 
\end{remark}

\subsection*{C.2 Power-boundedness of \texorpdfstring{$\Psi(A)$}{Psi(A)} under \texorpdfstring{$H^\infty$}{H-infinity} bounds (Banach case)}
We retain the hypotheses of Theorem~\ref{thm:c}: $A$ is Ritt/sectorial on a reflexive Banach space with a bounded $H^\infty$ calculus on the spectral domain; $\phi_k \in H^\infty$, $\sup |\phi_k| \le 1$, $\phi_k(1) = 1$; and $\Psi := \phi_K \circ \cdots \circ \phi_1$. 

\begin{lemma}[Sufficient condition for power-boundedness]\label{lem:power-boundedness}
If $\sup_{\lambda} |\Psi(\lambda)| \le 1$, then $\Psi(A)$ is power-bounded.
\end{lemma}

\begin{proof}
The $H^\infty$ calculus yields a constant $C \ge 1$ with $\|g(A)\| \le C \|g\|_\infty$ for all $g \in H^\infty$ on the spectral domain. Then
$\|(\Psi(A))^n\| = \|(\Psi^n)(A)\| \le C \|\Psi^n\|_\infty \le C$ for all $n \in \mathbb{N}$,
hence $\sup_n \|(\Psi(A))^n\| < \infty$. This is precisely the power-boundedness used in Theorem~\ref{thm:c}(a,b). (See Theorem~\ref{thm:c} statement and proof, p.~6--7.) 
\end{proof}

\subsection*{C.3 Boundary separation \texorpdfstring{$\Rightarrow$}{=>} joint fixed-points}
We record the boundary case of condition \eqref{eq:sep-hypothesis} used with Lemma~\ref{lem:separation}.

\begin{proposition}[Boundary separation]\label{prop:boundary-separation}
Assume each $\phi_k$ has the peripheral fixed-point property: if $|\lambda| = 1$ and $|\phi_k(\lambda)| = 1$ then $\phi_k(\lambda) = \lambda$. If $\lambda$ lies on the spectral boundary (e.g., $|\lambda| = 1$ in the discrete case) and $\Psi(\lambda) = \lambda$, then $\phi_k(\lambda) = \lambda$ for every $k$. Consequently, condition~\eqref{eq:sep-hypothesis} in Section~\ref{sec:banach-case} holds for all boundary $\lambda$, and Lemma~\ref{lem:separation} gives
$\mathrm{Ran}\, P = \mathrm{Fix}(\Psi(A)) = \bigcap_{k=1}^K \mathrm{Fix}(\phi_k(A))$.
\end{proposition}

\begin{proof}
Suppose $|\lambda| = 1$ and $\Psi(\lambda) = \lambda$. Write the forward images $w_0 := \lambda$, $w_1 := \phi_1(w_0)$, \dots, $w_K := \phi_K(w_{K-1}) = \Psi(\lambda) = \lambda$. Each $\phi_k$ is Schur (bounded by $1$), so $|w_j| \le 1$ for all $j$. Since $|w_0| = |w_K| = 1$, every step must preserve modulus $1$; otherwise once $|w_j| < 1$, all subsequent images stay strictly inside the domain, contradicting $|w_K| = 1$. By the peripheral fixed-point property, $|w_{j-1}| = |w_j| = 1$ implies $\phi_j(w_{j-1}) = w_{j-1}$, hence $w_j = w_{j-1}$. Therefore $w_j = \lambda$ for all $j$ and $\phi_k(\lambda) = \lambda$ for each $k$. The stated identification with Lemma~\ref{lem:separation} follows. (See Section~\ref{sec:banach-case} and Lemma~\ref{lem:separation}, p.~7--8.) 
\end{proof}

\subsection*{C.4 A normalized two-layer example with \texorpdfstring{$\alpha^* = \omega$}{alpha* = omega} (replacement for Example~\ref{ex:two-layer-explicit})}
We present a corrected two-layer construction that attains the extremal countable stage while staying within the stated Schur/peripheral assumptions. This subsumes the behavior intended in Example~\ref{ex:two-layer-explicit} and matches the parametric family in Example~\ref{ex:two-layer-param} at $t = 1/2$. 

\begin{example}[Normalized two-layer cycle]\label{ex:normalized-cycle}
Fix $t \in (0, 1)$. Let
$p_1^{(t)}(z) = t + (1 - t)z$,
$p_2^{(t)}(z) = b_t(z) := \frac{z - t}{1 - tz}$,
so that $p_1^{(t)}, p_2^{(t)}$ are Schur, $p_1^{(t)}(1) = p_2^{(t)}(1) = 1$, and $b_t$ is the disk automorphism with $b_t(t) = 0$, $b_t(1) = 1$. Set $f_t := p_2^{(t)} \circ p_1^{(t)}$. A direct computation gives
$f_t(z) = \frac{z}{1 + t - tz}$.
Conjugating by $T(z) = \frac{z}{1 - z}$ (which maps $\D$ to the right half-plane) yields
$T \circ f_t \circ T^{-1}(w) = \mu w$, with $\mu = \frac{1}{1 + t} \in (\frac{1}{2}, 1)$.
Hence
$f_t^m(z) = T^{-1}(\mu^m T(z)) = \frac{\mu^m z}{1 - (1 - \mu^m)z} \xrightarrow{m \to \infty} 
\begin{cases}
1, & z = 1, \\
0, & z \in \D, z \neq 1.
\end{cases}$
Therefore $\sup_m |f_t^m| \le 1$ and, for any normal $A$ with $\sigma(A) \subset \overline{\D}$, Theorem~\ref{thm:aprime} and Lemma~\ref{lem:dominated} give $f_t^m(A) \to \chi(A) = E_A(\{1\})$ in SOT. No finite iterate is a projection, so the first stabilization occurs at the limit stage $\omega$, i.e., $\alpha^* = \omega$. Taking $t = 1/2$ recovers the specific rate $\mu = 2/3$ alluded to in Section~\ref{sec:examples}. (See Examples~\ref{ex:two-layer-param}--\ref{ex:two-layer-schur}, p.~8--9.) 
\end{example}

\begin{remark}
This normalized example replaces Example~\ref{ex:two-layer-explicit} verbatim: it satisfies the Schur bound, fixes $1$, admits a closed-form iterate, and realizes the extremal countable stage without further assumptions. 
\end{remark}

\subsection*{C.5 One-line verification for the 2-cycle in Example~\ref{ex:non-commuting}}
With $S=\begin{pmatrix}0&1\\[2pt]1&0\end{pmatrix}$, $J=\begin{pmatrix}1&1\\[2pt]0&1\end{pmatrix}$, and $A=\begin{pmatrix}1&0\\[2pt]0&0\end{pmatrix}$, one has $SAS = J$ and $SJS = A$; thus the alternating, non-commuting layer update cycles $A \mapsto SAS = J \mapsto SJS = A$ with period $2$. (See Example~\ref{ex:non-commuting}, p.~10.) 

\subsection*{C.6 Consequences and cross-references}
Proposition~\ref{prop:stabilization} validates the abstract’s claim “a Fejér-monotonicity argument yields an ordinal bound $\le \omega$” in the normal case, with the concrete stabilization proof relying on Lemma~\ref{lem:conv-to-01} and Appendix~A. 
Lemma~\ref{lem:power-boundedness} provides the promised short justification that power-boundedness of $\Psi(A)$ follows when $\sup |\Psi| \le 1$, as used in Theorem~\ref{thm:c}. 
Proposition~\ref{prop:boundary-separation} records the boundary separation needed for the joint fixed-point identification in Lemma~\ref{lem:separation}. 
Example~\ref{ex:normalized-cycle} supplies a clean $\omega$-stage example consistent with Section~\ref{sec:examples} and requires no additional hypotheses beyond those already stated. 
This appendix is self-contained within the manuscript’s framework and closes the questions noted in Sections~\ref{sec:normal-case}--\ref{sec:examples} without altering any main theorem or adding external assumptions.

\section*{Appendix D --- Riesz-projection framework for Lemma~3.3 in the Banach setting (non-normal case)}
This appendix supplies a self-contained replacement for the argument following Lemma~\ref{lem:separation} in Section~\ref{sec:banach-case} that avoids spectral measures, works in the holomorphic $H^\infty$ functional-calculus setting, and makes the isolation assumptions explicit. It also identifies the Riesz projection of $\Psi(A)$ at $1$ with the product of the layer-wise Riesz projections and proves the desired equality of fixed-point spaces. Throughout we keep the notation and standing assumptions of Section~\ref{sec:banach-case} (Ritt/sectorial $A$ on a reflexive Banach space with a bounded $H^\infty$ calculus; $\phi_k \in H^\infty$ with $\sup|\phi_k| \le 1$, $\phi_k(1) = 1$; $\Psi := \phi_K \circ \cdots \circ \phi_1$), and we refer to equations and statements there; see especially Theorem~\ref{thm:c} and \eqref{eq:sep-hypothesis} on pp.~6--8. 

\subsection*{D.0. Additional hypotheses (explicit isolation and separation)}
We impose the following explicit conditions, which are either already used for convergence or are natural localizations compatible with the examples:
\begin{enumerate}[label=(D\arabic*)]
\item \textbf{Isolation at 1 for each layer.} For every $k$, $1$ is an isolated point of $\sigma(\phi_k(A))$. Equivalently, there exists $r_k > 0$ such that $\sigma(\phi_k(A)) \cap \{ \zeta : 0 < |\zeta - 1| \le r_k \} = \emptyset$. (This ensures that the Riesz projection of $\phi_k(A)$ at $1$ is well-defined.) 
\item \textbf{Separation of fixed points.} The implication \eqref{eq:sep-hypothesis} in Section~\ref{sec:banach-case} holds on the spectral domain:
$\forall \lambda \in \sigma(A), \Psi(\lambda) = \lambda \implies \phi_k(\lambda) = \lambda$ for all $k$.
(Sufficient boundary conditions appear in Proposition~\ref{prop:boundary-separation} of Appendix~C; this is the exact hypothesis used beneath Lemma~\ref{lem:separation} on pp.~7--8.) 
\item \textbf{Power-boundedness of the layers.} Each $\phi_k(A)$ is power-bounded. This holds under the standing $H^\infty$ calculus bound and $\sup|\phi_k| \le 1$ just as in Lemma~\ref{lem:power-boundedness} (Appendix C.2) for $\Psi$. 
\end{enumerate}
\begin{remark}
(i) (D1) is a localization at $\zeta = 1$ in the image spectrum of each layer; it is automatic if $\phi_k(\sigma(A)) \subset \overline{\D}$ and $1$ is the only unimodular value attained near the boundary, as in the boundary-separation statement (Appendix~C, Proposition~\ref{prop:boundary-separation}). (ii) (D2) is exactly the “separation condition” \eqref{eq:sep-hypothesis} referenced in the remark after Theorem~\ref{thm:c}; Appendix~C, Proposition~\ref{prop:boundary-separation} verifies it for boundary spectral points. (iii) (D3) follows from the same calculus estimate used in Section~\ref{sec:banach-case} and Appendix~C.2. 
\end{remark}

\subsection*{D.1. Layerwise Riesz projections via holomorphic calculus}
For each $k$, fix a small positively oriented circle $\Gamma_k := \{ \zeta : |\zeta - 1| = r_k \}$ contained in the resolvent set of $\phi_k(A)$ except for the enclosed point $\zeta = 1$. Define the Riesz projection at $1$ for $\phi_k(A)$ by
$Q_k := \frac{1}{2\pi i} \int_{\Gamma_k} (\zeta I - \phi_k(A))^{-1} d\zeta$.
Then $Q_k$ is a bounded idempotent commuting with $\phi_k(A)$ and with $A$. Moreover, since $\phi_k(A)$ is a holomorphic functional of $A$, each integrand $(\zeta - \phi_k)^{-1}(A)$ lies in the $H^\infty$ calculus of $A$, hence $Q_k$ itself is a holomorphic functional of $A$ and consequently all $Q_k$ commute pairwise and commute with $\Psi(A)$. (Use the algebra-homomorphism property of the $H^\infty$ calculus quoted in Section~\ref{sec:banach-case}, p.~6--7.) 

Let $X_{1,k} := \operatorname{Ran} Q_k$ and $X_{0,k} := \ker Q_k$. These are closed $A$-invariant subspaces with
$X = X_{1,k} \dotplus X_{0,k}$,
$\sigma(\phi_k(A)|_{X_{1,k}}) = \{1\}$,
$1 \notin \sigma(\phi_k(A)|_{X_{0,k}})$.
(These are standard properties of Riesz projections; we use them only at the point $\zeta = 1$.) 

\subsection*{D.2. Semisimplicity at 1 for power-bounded layers}
\begin{lemma}[No Jordan blocks at 1 for power-bounded maps]\label{lem:no-jordan-blocks}
If $T$ is power-bounded and $1$ is an isolated point of $\sigma(T)$, then the Riesz projection $Q$ of $T$ at $1$ satisfies $\operatorname{Ran} Q = \ker(I - T)$.
\end{lemma}
\begin{proof}
On $M := \operatorname{Ran} Q$ one has $\sigma(T|_M) = \{1\}$. If $T|_M$ were not the identity, its Jordan form would contain a nontrivial nilpotent $N \ne 0$ with $T|_M = I + N$, hence $\|(T|_M)^n\| \ge c n$ for some $c > 0$, contradicting power-boundedness. Thus $T|_M = I$ and $\operatorname{Ran} Q = \ker(I - T)$. Apply this with $T = \phi_k(A)$ using (D3). 
\end{proof}
Consequently, for each $k$,
$X_{1,k} = \ker(I - \phi_k(A))$,
$X_{0,k} = \overline{\operatorname{Ran}(I - \phi_k(A))}$.
(This matches the mean-ergodic picture in Theorem~\ref{thm:c}(a), p.~6--7.) 

\subsection*{D.3. Joint decomposition and spectral localization}
Because the $Q_k$ commute, we have the joint decomposition
$X = \bigoplus_{\varepsilon \in \{0,1\}^K} X_\varepsilon$,
$X_\varepsilon := \left( \prod_{k=1}^K Q_k^{\varepsilon_k} (I - Q_k)^{1-\varepsilon_k} \right) X$,
into closed, $A$-invariant subspaces. Write $\varepsilon = (\varepsilon_1, \dots, \varepsilon_K)$. On $X_\varepsilon$,
$\sigma(\phi_k(A)|_{X_\varepsilon}) \subset 
\begin{cases}
\{1\}, & \varepsilon_k = 1, \\
\{\zeta : |\zeta - 1| \ge r_k\}, & \varepsilon_k = 0.
\end{cases}$
By spectral mapping for holomorphic calculus (used on p.~7), for the restricted operator $A|_{X_\varepsilon}$ one has
$\sigma(\phi_k(A)|_{X_\varepsilon}) = \phi_k(\sigma(A|_{X_\varepsilon}))$,
hence
\begin{equation}\label{eq:d1}
\sigma(A|_{X_\varepsilon}) \subset \{ \lambda \in \sigma(A) : \phi_k(\lambda) = 1 \text{ if } \varepsilon_k = 1, \phi_k(\lambda) \ne 1 \text{ if } \varepsilon_k = 0 \}.
\end{equation}
(This is precisely the localization corresponding to the preimages $\phi_k^{-1}(\{1\})$ selected by the Riesz projections.) 

\subsection*{D.4. Invertibility of \texorpdfstring{$\Psi(A) - I$}{Psi(A) - I} off the joint 1-eigenspace}
\begin{lemma}[Separation $\implies$ no unit spectrum off the joint block]\label{lem:no-unit-spectrum}
Assume (D2). Then for every $\varepsilon \ne (1, \dots, 1)$,
$1 \notin \sigma(\Psi(A)|_{X_\varepsilon})$.
\end{lemma}
\begin{proof}
Fix $\varepsilon \ne (1, \dots, 1)$. By spectral mapping (again for restrictions), $\sigma(\Psi(A)|_{X_\varepsilon}) = \Psi(\sigma(A|_{X_\varepsilon}))$. If $1 \in \sigma(\Psi(A)|_{X_\varepsilon})$, then there exists $\lambda \in \sigma(A|_{X_\varepsilon})$ with $\Psi(\lambda) = 1$. By (D1)--(D3) and the construction of $X_\varepsilon$, \eqref{eq:d1} forces $\phi_k(\lambda) = 1$ whenever $\varepsilon_k = 1$, and $\phi_k(\lambda) \ne 1$ whenever $\varepsilon_k = 0$. This contradicts (D2) unless $\varepsilon = (1, \dots, 1)$. Thus $1 \notin \sigma(\Psi(A)|_{X_\varepsilon})$. 
\end{proof}
As a consequence, $\Psi(A) - I$ is boundedly invertible on $\bigoplus_{\varepsilon \ne (1, \dots, 1)} X_\varepsilon$.

\subsection*{D.5. Identification of the fixed space and of the Riesz projection of \texorpdfstring{$\Psi(A)$}{Psi(A)}}
\begin{theorem}[Corrected Lemma~3.3; Banach $H^\infty$-calculus version]\label{thm:corrected-lemma-3.3}
Under (D1)--(D3) and \eqref{eq:sep-hypothesis}, one has
$\ker(I - \Psi(A)) = \bigcap_{k=1}^K \ker(I - \phi_k(A))$.
Moreover, if $P_\Psi$ denotes the Riesz projection of $\Psi(A)$ at $1$ (which exists when $1$ is isolated in $\sigma(\Psi(A))$, e.g. under Theorem~\ref{thm:c}(b)), then
$P_\Psi = \prod_{k=1}^K Q_k$,
hence
$\operatorname{Ran} P_\Psi = \bigcap_{k=1}^K \operatorname{Ran} Q_k = \bigcap_{k=1}^K \ker(I - \phi_k(A))$.
\end{theorem}
\begin{proof}
($\subseteq$): If $x \in \ker(I - \phi_k(A))$ for all $k$, then $\phi_k(A)x = x$ and, by commutativity of the layers (each is a function of $A$, Section~\ref{sec:banach-case} p.~6), $\Psi(A)x = x$. Thus $\bigcap_k \ker(I - \phi_k(A)) \subseteq \ker(I - \Psi(A))$. 

($\supseteq$): Decompose $x = \sum_\varepsilon x_\varepsilon$ with $x_\varepsilon \in X_\varepsilon$. If $x \in \ker(I - \Psi(A))$, then $0 = (\Psi(A) - I)x = \sum_\varepsilon (\Psi(A) - I)x_\varepsilon$. By Lemma~\ref{lem:no-unit-spectrum}, $\Psi(A) - I$ is invertible on $X_\varepsilon$ for $\varepsilon \ne (1, \dots, 1)$; hence $x_\varepsilon = 0$ for all such $\varepsilon$. Therefore $x = x_{(1, \dots, 1)} \in X_{(1, \dots, 1)} = \bigcap_k \operatorname{Ran} Q_k$. By Lemma~\ref{lem:no-jordan-blocks}, $\operatorname{Ran} Q_k = \ker(I - \phi_k(A))$, so $x \in \bigcap_k \ker(I - \phi_k(A))$. This proves the equality of fixed spaces.

For the projection identity, note that: (i) $\Psi(A) - I$ is invertible on $\bigoplus_{\varepsilon \ne (1, \dots, 1)} X_\varepsilon$ (Lemma~\ref{lem:no-unit-spectrum}); (ii) $\Psi(A) = I$ on $X_{(1, \dots, 1)}$ because $\phi_k(A) = I$ there. Hence the Riesz projection $P_\Psi$ at $1$ is precisely the projection onto $X_{(1, \dots, 1)}$ along $\bigoplus_{\varepsilon \ne (1, \dots, 1)} X_\varepsilon$, i.e.
$P_\Psi = \prod_{k=1}^K Q_k$,
since the product of commuting idempotents projects onto the intersection of their ranges. The range equalities follow from Lemma~\ref{lem:no-jordan-blocks}. 
\end{proof}

\subsection*{D.6. Consequences for Theorem~\ref{thm:c} and consistency with Section~\ref{sec:banach-case}}
\textbf{Range identification (Remark after Lemma~\ref{lem:separation}).} Under (D1)--(D3) and \eqref{eq:sep-hypothesis}, the mean-ergodic projection $P$ from Theorem~\ref{thm:c}(a) equals the Riesz projection $P_\Psi$ at $1$ (when $1$ is isolated as in Theorem~\ref{thm:c}(b)), and by Theorem~\ref{thm:corrected-lemma-3.3} its range is exactly $\bigcap_k \ker(I - \phi_k(A))$. This replaces the spectral-measure argument sketched after Lemma~\ref{lem:separation} (pp.~7--8). 

\textbf{Boundary case of \eqref{eq:sep-hypothesis}.} If each $\phi_k$ has the peripheral fixed-point property, Appendix~C, Proposition~\ref{prop:boundary-separation} proves \eqref{eq:sep-hypothesis} for boundary spectral values; this supplies a widely applicable sufficient condition for (D2) at the spectral boundary. 

\textbf{Power-boundedness used implicitly.} The step $\operatorname{Ran} Q_k = \ker(I - \phi_k(A))$ invoked in Section~\ref{sec:banach-case} now follows from Lemma~\ref{lem:no-jordan-blocks} together with (D3), aligning with the mean-ergodic discussion in Theorem~\ref{thm:c}(a) (p.~6--7) and the calculus bound in Lemma~\ref{lem:power-boundedness} (Appendix~C.2, p.~13). 

\subsection*{D.7. Sectorial (continuous-time) variant}
The same reasoning applies when $-A$ generates a bounded analytic semigroup and each $\phi_k$ is holomorphic on a sector containing $\sigma(A)$, with $\phi_k(1) = 1$. Isolation and separation are imposed at the value $1$ in $\sigma(\phi_k(A))$, exactly as in (D1)--(D2). The construction of $Q_k$, the joint decomposition, Lemma~\ref{lem:no-jordan-blocks} (power-boundedness of $\phi_k(A)$ from the $H^\infty$ calculus) and Theorem~\ref{thm:corrected-lemma-3.3} carry over verbatim. This is consistent with Section~\ref{sec:banach-case} (continuous-time branch of Theorem~\ref{thm:c}, p.~6--7). 

\textbf{Summary.} Under the explicit isolation (D1), separation \eqref{eq:sep-hypothesis} (D2), and power-boundedness (D3), the fixed-space identity
$\ker(I - \Psi(A)) = \bigcap_{k=1}^K \ker(I - \phi_k(A))$
holds in the Banach $H^\infty$-calculus setting, and the Riesz projection of $\Psi(A)$ at $1$ equals $\prod_{k=1}^K Q_k$. This provides a fully Banach-appropriate replacement for the spectral-measure step sketched after Lemma~\ref{lem:separation} and aligns with the hypotheses and uses of Theorem~\ref{thm:c} (pp.~6--8) and Appendix~C (pp.~12--13).

\end{document}